\documentclass[a4paper,12pt]{amsart}

\usepackage{amsfonts}
\usepackage{amsmath}
\usepackage{amssymb}
\usepackage{mathrsfs}
\usepackage[colorlinks]{hyperref}

\usepackage{graphicx}

\usepackage[usenames]{color}

\newtheorem{thm}{Theorem}[section]

\numberwithin{equation}{section}

\theoremstyle{definition}
\newtheorem{definition}[thm]{Definition}
\newtheorem{rem}[thm]{Remark}
\newtheorem{example}[thm]{Example}

\begin{document}

\title[Fractional Cauchy problems in $q$-calculus]
{Fractional Cauchy problems associated with the bi-ordinal Hilfer  fractional $q$-derivative}

\author[Erkinjon Karimov]{Erkinjon Karimov}
\address{
  Erkinjon Karimov:
  \endgraf
  Fergana State University
  \endgraf
 19 Murabbiylar str., Fergana, 140100
  \endgraf
 Uzbekistan
 \endgraf
   and
  \endgraf
  V.I.Romanovskiy Institute of Mathematics
  \endgraf
 9 Universitet str., Tashkent, 100174
 \endgraf
 Uzbekistan
  \endgraf
 {\it E-mail address} {\rm erkinjon@gmail.com}
  }

\author[Michael Ruzhansky]{Michael Ruzhansky}
\address{
 Michael Ruzhansky:
  \endgraf
Department of Mathematics: Analysis, Logic and Discrete Mathematics,
  \endgraf
 Ghent University, Ghent,
 \endgraf
  Belgium 
  \endgraf
  and 
 \endgraf
 School of Mathematical Sciences, Queen Mary University of London, London,
 \endgraf
 UK  
 \endgraf
  {\it E-mail address} {\rm michael.ruzhansky@ugent.be}
  }

\author[Serikbol Shaimardan]{Serikbol Shaimardan}
\address{
  Serikbol Shaimardan:
  \endgraf
  L. N. Gumilyov Eurasian National University, Astana,
  \endgraf
  Kazakhstan 
  \endgraf
  and 
  \endgraf
Department of Mathematics: Analysis, Logic and Discrete Mathematics
  \endgraf
 Ghent University, Ghent,
 \endgraf
  Belgium
  \endgraf
  {\it E-mail address} {\rm shaimardan.serik@gmail.com} 
  }

\thanks{The authors are supported by the FWO Odysseus 1 grant G.0H94.18N: Analysis and Partial Differential Equations and by the Methusalem programme of the Ghent University Special Research Fund (BOF) (Grant number 01M01021). Michael Ruzhansky is also supported by EPSRC grant EP/R003025/2, and  the third authors by the international internship program “Bolashak” of the Republic of Kazakhstan.}

\date{}

\begin{abstract}
The aim of this paper is the investigation of the existence and uniqueness of solutions to Cauchy-type problems for fractional $q$-difference equations with the bi‐ordinal Hilfer fractional $q$-derivative which is  an extension of the Hilfer fractional $q$-derivative. An approach is based on the equivalence of the nonlinear Cauchy-type problem with a nonlinear Volterra $q$-integral equation of the second kind.  Applying an analog of Banach’s fixed point theorem we prove the uniqueness and the existence of the solution. Moreover, we present an explicit solution to the $q$-analog of the Cauchy problem for the linear case. 
\end{abstract}
 


\maketitle

\section{Introduction}
 The quantum groups provide the key to $q$-deforming the basic structures of physics from the point of view of non-commutative geometry. The theory plays an important role in the conformal field theory, exact soluble models in statistical physics \cite{Baxte1982} and in a wide range of applications, from cosmic strings and black holes to the solid state physics problems \cite{Anyons1992}, \cite{BD1999}, \cite{LSN2006}, \cite{Wilczek1990}. For instance, the $q$-deformed Lie algebras are extended versions of the usual Lie algebras   which are beyond the scope of usual Lie algebras \cite{BD1999}. Moreover, it has recently led to  the  application of the $q$-calculus in the construction of generalized statistical mechanics where non-extensivity properties can arise from the $q$-deformed theory \cite{GT2004}.  For some analytical constructions on  quantum groups,  we refer to  \cite{AMR2018}.

Recently, mathematicians have paid much attention to the $q$-calculus \cite{CK2000}, \cite{E2012}, \cite{E2002} and fractional $q$-differential equations \cite{A2014}, \cite{A1969}, \cite{A1966}, \cite{PMS2007}, \cite{SPT2020}, \cite{PMS2009}, \cite{Shaimardan16}. The $q$-calculus (or quantum calculus)  can be dated back to  the beginning of the twentieth century, to F. H. Jackson and R. D. Carmichael's work \cite{C1912}, \cite{J1908}, \cite{J1910}. There have been some papers dealing with the existence and uniqueness or multiplicity of solutions for nonlinear fractional $q$-difference equations by the use of some well-known fixed point theorems \cite{KF1968}. For some recent developments on the subject, see e.g. \cite{AM2012},  \cite{F2011}, \cite{F2010}, \cite{ZCZ2013} and the references therein.

The notations used in this introduction are explained in Section 2 below. In this paper,  we also focus more on the $q$-analogue of Hilfer’s bi-ordinal   fractional derivative or composite fractional derivative operator (see   Definition \ref{definition3.1}) and we derive sufficient conditions for the existence of a unique solution of the  Cauchy-type $q$-fractional problem \eqref{additive3.3}-\eqref{additive3.4}.  Moreover, for the proof of this theorem, we prove an equivalence theorem (Theorem \ref{thm3.4}) of independent interest. In the case when $\mu = 0$ this is the generalized Riemann–Liouville   fractional  $q$-derivative (see \eqref{additive2.1}) and in case when $\mu = 1$ it corresponds to the Caputo fractional $q$-derivative (see \cite{PMS2009}).

The paper is organized as follows: The main results are presented and proved in Section \ref{section3}   and the announced examples are given in Section \ref{section4}. In the final section, we present the explicit solution to the Cauchy problem for the linear case, applying the successive-iteration method. In order to facilitate presentations, we include in Section \ref{section2}  some necessary preliminaries.

\section{Preliminaries}\label{section2}
First, we recall several concepts of the $q$-calculus that we need in the paper. We assume that $0<q<1$.

Let $\alpha \in \mathbb{R}$. Then a $q$-real number $[\alpha ]_q$ is defined by (\cite{CK2000})
$$
[\alpha ]_{q}:=\frac{1-q^{\alpha }}{1-q},
$$
and the $q$-binomial coefficient $[n]_{q}!$ is defined by (\cite{CK2000})
\begin{equation*}
[n]_{q}!:=\left\{
\begin{array}{l}
{1,\mathrm{\;\;\;\;\;\;\;\;\;\;\;\;\;\;\;\;\;\;\;\;\;\;\;\;\;\;\;\;\;\;\;\;
\;if\;}n={0}}, \\
{[1]_{q}\times [2]_{q}\times \cdots \times [n]_{q},
\mathrm{\;if\;}\,n\in \mathbb{N}}.
\end{array}
\right.
\end{equation*}

The $q$-shifted
operations for $n\in\mathbb{N}$ are defined by 
\begin{eqnarray*}
(a;q)_0=1,  \; (a;q)_n=\prod\limits_{k=0}^{n-1}\left(1-q^ka\right),  \; (a;q)_\infty = \lim\limits_{n\rightarrow\infty}(a;q)_n.
\end{eqnarray*}

Moreover,
\begin{eqnarray}\label{additive2.1C} 
(a;q)_\alpha=\frac{(a;q)_\infty}{(q^\alpha a;q)_\infty}.
\end{eqnarray}

The $q$-gamma function $\Gamma_q(x)$ is defined by (\cite{CK2000})
\begin{eqnarray}\label{additive2.1D} 
\Gamma_q(x)=\frac{(q; q)_\infty }{(q^x; q)_\infty }(1-q)^{1-x} 
\end{eqnarray}
for any $x > 0$. This  $q$-Gamma function has the following property (\cite{CK2000}): 
\begin{eqnarray*} 
\Gamma_q(x)[x]_q=\Gamma_q(x+1).
\end{eqnarray*}

The $q$-derivative of a function $f$ with respect to $x$ is defined by (\cite{J1908})
\begin{eqnarray*}
D_{q}f(x)=\frac{f(x)-f(qx)}{x(1-q)},
\end{eqnarray*}
and the $q$-derivatives $D^{n}_q(f(x))$ of higher order are defined inductively as follows (\cite{PMS2007}):
\begin{eqnarray*}
D_q^0(f(x)):=f(x), \;\;\; D^{n}_q(f(x)):=D_q\left(D_q^{n-1}f(x)\right), (n=1,2,3, \dots).
\end{eqnarray*}

The general $q$-integral (or Jackson integral)   is defined    (\cite{J1910}) as
\begin{eqnarray*} 
\int\limits_a^b f(x)d_{q}x=\int\limits_0^b f(x)d_{q}x-
\int\limits_0^a f(x)d_{q}x,
\end{eqnarray*}
for $0<a<b$, where 
\begin{eqnarray*} 
\int\limits_0^a f(x)d_{q}x=(1-q)a\sum\limits_{m=0}^\infty q^{m}f(aq^{m}). 
\end{eqnarray*}

In the middle of the last century, W. A. Al-Salam \cite{A1969} and  R. P. Agarwal \cite{A1966}  presented a $q$-analog of the Riemann–Liouville fractional integral  with only zero as a lower limit of the $q$-integration and a fractional $q$-derivative. After that, P. M. Rajkovic', S. D. Marinkovic' and M. S. Stankovic' \cite{PMS2007} (see also \cite{PMS2009}) allowed the lower limit of the $q$-integration   to be nonzero and reintroduced the Riemann-Liouville $q$-fractional integrals $I^\alpha_{q,a+}f$ of order $\alpha>0$ by 
\begin{eqnarray}\label{additive2.1}
\left(I^\alpha_{q,a+}f\right)(x)=\frac{x^{\alpha-1}}{\Gamma_q(\alpha)}\int\limits_a^x(qt/x;q)_{\alpha-1}f(t)d_qt.
\end{eqnarray}

The Riemann-Liouville fractional $q$-derivative $D^\alpha_{q,a+}f$ of order $\alpha>0$ is defined by 
\begin{eqnarray*} 
\left(D^\alpha_{q,a+}f\right)(x)=\left( D^{[\alpha]}_{q,a+}I^{[\alpha]-\alpha}_{q,a+}f\right)(x),
\end{eqnarray*}
where $[\alpha]$ denotes the smallest integer greater or equal to $\alpha$.

For $\alpha>0$ and $\lambda>-1$ the following identities hold \cite[Lemma 12]{PMS2009}:
\begin{eqnarray}\label{additive2.2}
\left(I^\alpha_{q,a+}t^\lambda(t/a;q)_\lambda\right)(x)=
\frac{\Gamma_q(\lambda+1)}{\Gamma_q(\alpha+\lambda+1)} x^{\alpha+\lambda}(x/a;q)_{\alpha+\lambda} 
\end{eqnarray}
and  
\begin{eqnarray}\label{additive2.3}
\left(D^\alpha_{q,a+}t^\lambda(t/a;q)_\lambda\right)(x)=
\frac{\Gamma_q(\lambda+1)}{\Gamma_q(\lambda+1-\alpha)} x^{\lambda-\alpha}(x/a;q)_{\lambda-\alpha},\;\;\;\alpha\leq\lambda, 
\end{eqnarray}
and 
\begin{eqnarray}\label{additive2.3a}
\left(D^\alpha_{q,a+}t^\lambda(t/a;q)_\lambda\right)(x)=0,\;\;\;\alpha>\lambda. 
\end{eqnarray}
 
For  all further discussions, we assume that the functions are defined on an interval $[0, b]$ and $b>0$. Moreover,  $a \in[0, b]$ is an arbitrary fixed point. In \cite[Subsection 4.3]{AM2012}, M. H. Annaby and Z. S. Mansour   presented definitions of spaces of $q$-integrable, $q$-absolutely continuous on $[0, a]$ functions.  We also give characterizations of those modified spaces $[a, b]$ which will be used later. For $1\leq p<\infty$ we  define the space $L^p_q=L^p_q [a, b]$ by 
\begin{eqnarray*}
L^p_q [a, b] :=\left\{f:[a,b]\rightarrow \mathbb{ R}:\left(\int\limits_a^b |f(x)|^pd_qx\right)^\frac{1}{p}<\infty\right\}.
\end{eqnarray*}

 Let $\alpha>0$, $\beta>0$ and $1\leq p<\infty$. Then the $q$–fractional integrals have the following
semi-group property (see \cite[Lemma 2.3]{SPT2020}):
\begin{eqnarray}\label{additive2.4}
\left(I^\alpha_{q,a+}I^\beta_{q,a+}f\right)(x) =\left(I^{\alpha+\beta}_{q,a+}f\right)(x),
\end{eqnarray}
for all  $x \in[a, b]$ and $f \in L^p_q[a,b]$.

 Let $\alpha>\beta > 0$, $1 \leq p<\infty$ and $f \in L^p_q[a,b]$. Then the following (see  \cite[Lemma 2.3]{SPT2020})
equalities
\begin{eqnarray}\label{additive2.5}
\left(D^\alpha_{q,a+}I^\alpha_{q,a+}\right)(x)=f(x),\;\;\;\left(D^\beta_{q,a+}I^{\alpha }_{q,a+}f\right)(x)=\left(I^{\alpha-\beta}_{q,a+}f\right)(x),
\end{eqnarray}
hold for all  $x\in[a,b]$.
 
 A function  $f:[a,b]\rightarrow \mathbb{ R}$ is called $q$-absolutely continuous
if\, $\exists \varphi\in L^1_q[a,b]$ such that  
\begin{eqnarray*} 
f(x)=f(a)+\int\limits_a^x\varphi(t)d_qt
\end{eqnarray*}
for all $x\in[a,b]$  (see also, \cite[Definition 2.4]{SPT2020}).

The collection of all $q$-absolutely continuous functions on $[a,b]$ is denoted by $AC_q[a,b]$. For $n= {1,2,3,\dots}$ we denote by $AC^n_q[a, b]$ the space of real-valued functions $f(x)$ which have $q$-derivatives up to order $n-1$ on $[a, b]$ such
that $D_q^{n-1}f \in AC_q[a,b]$:
$$
AC^n_q[a, b]:=\left\{f:[a,b]\rightarrow \mathbb{ R}; D_q^{n-1}f \in AC_q[a,b] \right\}.
$$

The function $f:[a,b]\rightarrow \mathbb{ R}$ belongs to $AC^n_q[a, b]$ if the following equality holds (see \cite[Lemma 2.5]{SPT2020}):
\begin{eqnarray*} 
f(x) &=& \frac{x^{n-1}}{\Gamma_q\left(n\right)}\int\limits_a^x(qt/x;q)_{n-1}\varphi(t)d_qt+\sum\limits_{k=0}^{n-1}c_kx^k(a/x;q)_k,
\end{eqnarray*}
where $\varphi(x):=D^{n}_qf(x)$ and $c_k=\frac{D^{k}_qf(a)}{\Gamma_q(k)}, k=0,1,2,\dots, n-1$, are arbitrary constants.

Let $f \in L_q^1[a,b]$ and $ \left(I^{n-\alpha}_{q,a+}f\right) \in AC^n_q[a,b]$ with $n=[\alpha], \alpha>0$. Then the following equality holds (see \cite[Lemma 2.5]{SPT2020}):
\begin{eqnarray}\label{additive2.6}
\left(I^\alpha_{q,a+}D_{q,a+}^\alpha f\right)(x)=f(x)-\sum\limits_{k=1}^n\frac{\left(D^{\alpha-k}_{q,a+}f\right)(a)}{\Gamma_q\left(\alpha-k+1\right)}x^{\alpha-k}(a/x;q)_{\alpha-k}.
\end{eqnarray}

Let $\alpha>0$ and $1\leq p<\infty$.  Then the fractional integration operator $I^\alpha_{q,a+} $ is bounded in
$L^p_q [a, b]$ (see \cite[Lemma 2.6]{SPT2020}):
\begin{eqnarray}\label{additive2.7}
\|I^\alpha_{q,a+}f\|_{L^p_q [a, b]} &\leq& K\|f\|_{L^p_q [a, b]},
\end{eqnarray}
where $K=\frac{(b-qa)_q^\alpha}{ \Gamma_q(\alpha+1)}$ and $(b-qa)_q^\alpha=b^\alpha(qa/b;q)_\alpha$.

\section{Main results}\label{section3}

Now, we start to present the concept of the bi-ordinal fractional $q$-derivative generalizing the well-known Hilfer’s $q$-derivative (see \cite{H2000}, \cite{H2002} and \cite{T2012}). Bi-ordinal Hilfer derivative was introduced for the first time by Bulavatsky in \cite{B2014} in the case when the  fractional order lies between 0 and 1. Later, in \cite{K2021} it was generalized for higher fractional orders. Boundary-value problems for PDEs involving this fractional derivative were considered in recent works \cite{KTR2022} and \cite{T2022}.

Below for the first time, we introduce $q$-analogue of this integral-differential operator together with the key property which allows us to solve the Cauchy problem for the $q$-fractional differential equation.

\begin{definition}\label{definition3.1}  We define the bi-ordinal Hilfer   fractional $q$-derivative      $\mathcal{D}^{(\alpha,\beta)\mu}_{q,a+}$
  of orders $\alpha$ and $\beta$, and  type $\mu$   with respect to $x$ by
\begin{eqnarray}\label{additive3.1}
\left(\mathcal{D}^{(\alpha,\beta)\mu}_{q,a+}f\right)(x) &:=& \left(I_{q,a+}^{\mu(1-\alpha)}D_q\left(I_{q,a+}^{(1-\mu)(1-\beta)}f\right)\right)(x),
\end{eqnarray}
for $0<\alpha , \beta\leq1$ and $0\leq\mu\leq{1}$.
\end{definition}

In the case when $\mu=0$, the  generalized  fractional $q$-derivative   (\ref{additive3.1}) would correspond to the the Riemann-Liouville fractional $q$-derivative of order $\beta$ (see (\ref{additive2.1})) and in the case when $\mu= 1$ it corresponds to the Caputo fractional $q$-derivative $\left(_cD^\alpha _{q,a+}f\right)(x)$ defined by (see \cite{PMS2009}):
$$
\left(_cD^\alpha _{q,a+}f\right)(x) := \left(I_{q,a+}^{1-\alpha}D_qf\right)(x)=\frac{x^{-\alpha}}{\Gamma_q(1-\alpha)}\int\limits_a^x
\left(qt/x;q\right)_{-\alpha}f(t)d_qt.
$$

\begin{rem}\label{rem3.2}  Let $\gamma=\beta+\mu(n-\beta)$ and $\nu=\beta+\mu(\alpha-\beta)$ for  $n-1<\alpha,\beta \leq n, n\in \mathbb{N}$ and $0\leq \mu\leq 1$. The generalized  fractional $q$-derivative   $D_{q,a+}^{(\alpha,\beta)\mu} f$ can be represented as follows:
\begin{eqnarray}\label{additive3.2}
\left(D^{(\alpha,\beta)\mu}_{q,a+}f\right)(x) &:=& \left(I_{q,a+}^{\mu(n-\alpha)}D^n_q\left(I_{q,a+}^{(1-\mu)(n-\beta)}f\right)\right)(x)\nonumber\\
&=&\left(I_{q,a+}^{\gamma-\nu}D^n_q\left(I_{q,a+}^{n-\gamma}f\right)\right)(x)\nonumber\\
&=&\left(I_{q,a+}^{\gamma-\nu}D^{\gamma}_{q,a+} f\right)(x).
\end{eqnarray}
\end{rem}

Next, we study the question of the equivalence between the following Cauchy-type $q$-fractional problem  
\begin{eqnarray}\label{additive3.3}
\left(D^{(\alpha,\beta)\mu}_{q,a+}y\right)(x) &=& f\left(x,y(x)\right),\;\;\;n-1<\alpha\leq n; n\in\mathbb{N}, 0\leq \beta\leq 1,
\end{eqnarray}
\begin{eqnarray}\label{additive3.4}
\lim\limits_{x\rightarrow a+}\left(D^k_qI^{(1-\mu)(n-\beta)}_{q,a+}y \right)(x)=\xi_k,  \;\;\;\xi_k\in\mathbb{R},  k=1,2,\dots n,
\end{eqnarray}
and the Volterra $q$-integral equation of the second kind:
\begin{eqnarray}\label{additive3.5}
y(x)=\sum\limits_{k=1}^n\frac{\xi_k}{\Gamma_q\left(\gamma-k+1\right)}x^{\gamma-k}(a/x;q)_{\gamma-k}+
\left(I^ \nu_{q,a+}f(\cdot,y(\cdot))\right)(x).
\end{eqnarray}

\begin{thm}\label{thm3.3}  We assume  $n-1<\alpha,\beta\leq n, n\in\mathbb{N}$, $0\leq \mu\leq 1$, $\gamma=(n-\alpha)(1-\beta)$, $\nu=\beta+\mu(\alpha-\beta)$.  Let $f(\cdot, \cdot):  [a , b] \times \mathbb{R}\rightarrow \mathbb{R}$ be a function such that $f(\cdot, y(\cdot)) \in L_q^1[a, b]$ for all $y \in L_q^1[a, b]$.Then $y$ satisfies   the relations (\ref{additive3.3})-(\ref{additive3.4}) if and only if\, $y$ satisfies   the integral equation (\ref{additive3.5}).
\end{thm}
\begin{proof}  We assume that  $y\in L^1_q[a, b]$ satisfies   the relations (\ref{additive3.3})-(\ref{additive3.4}). Since $f(x,y(x)) \in L_q^1[a, b]$   we get that   $ \left(D^{(\alpha,\beta)\mu}_{q,a+}y\right)\in L^1_q[a, b]$ for   $\gamma=\beta+\mu(n-\beta)$, $\nu=\beta+\mu(\alpha-\beta)$ together with $n-1 < \alpha,\beta \leq n, n \in \mathbb{N}$, $0 \leq \mu \leq 1$. According to (\ref{additive2.7}) we obtain that  $\left(I^\nu_{q,a+}f\right)\in L^1_q[a,b]$. Applying the integral operator $I^\nu_{q,a+}$ to both sides of (\ref{additive3.3})  we obtain 
\begin{eqnarray}\label{additive3.6}
 I^ \nu_{q,a+}D^{(\alpha,\beta)\mu}_{q,a+}y  (x)=\left(I^ \nu_{q,a+}f(\cdot,y(\cdot))\right)(x).
\end{eqnarray}

From  Remark  \ref{rem3.2} and (\ref{additive2.4}), (\ref{additive2.6}) it follows that
\begin{eqnarray}\label{additive3.7}
\left(I^ \nu_{q,a+}D^{(\alpha,\beta)\mu}_{q,a+}y \right)(x)&=&\left(I^ \nu_{q,a+}I^{\gamma-\nu}_{q,a+}
D^{\gamma}_{q,a+}y \right)(x)\nonumber\\
&=&\left( I^{\gamma}_{q,a+}
D^{\gamma}_{q,a+}y \right)(x)\nonumber\\
&=& y(x) -\sum\limits_{k=1}^n\frac{\left(D^{\gamma-k}_{q,a+}f\right)(a)}{\Gamma_q\left(\gamma-k+1\right)}x^{\gamma-k}(a/x;q)_{\gamma-k}\nonumber\\
&=&y(x) -\sum\limits_{k=1}^n\frac{\left(D^k_qI^{(1-\mu)(n-\beta)}_{q,a+}f\right)(a)}{\Gamma_q\left(\gamma-k+1\right)}x^{\gamma-k} (a/x;q)_{\gamma-k}. 
\end{eqnarray}

According to (\ref{additive3.4}), (\ref{additive3.6}) and (\ref{additive3.7})  we get
the equation (\ref{additive3.5}). The necessity is proved.

Now we prove the sufficiency. Let $y\in L^1_q[a, b]$  satisfy  equation (\ref{additive3.5}). Applying the operator $D^{(\alpha,\beta)\mu}_{q,a+}$  to both sides of (\ref{additive3.5}), we have that
\begin{eqnarray}\label{additive3.8}
\left(D^{(\alpha,\beta)\mu}_{q,a+}y\right)(x)&=&\sum\limits_{k=1}^n\frac{\xi_k}{\Gamma_q\left(\gamma-k+1\right)}
\left(D^{(\alpha,\beta)\mu}_{q,a+}t^{\gamma-k}(a/t;q)_{\gamma-k}\right)(x)\nonumber\\
&+& D^{(\alpha,\beta)\mu}_{q,a+}\left(I^ \nu_{q,a+}f(\cdot,y(\cdot))\right)(x).
\end{eqnarray}

From (\ref{additive2.3a})  it follows  that
\begin{eqnarray}\label{additive3.9}
\left(D^{\gamma}_{q,a+}t^{\gamma-k}(a/t;q)_{\gamma-k}\right)(x)=0,\;\;\;1\leq k\leq n.  
\end{eqnarray}

By  Remark \ref{rem3.2} and (\ref{additive2.4}), (\ref{additive2.5})  and (\ref{additive3.9}), we find that
\begin{eqnarray*}
\left(D^{(\alpha,\beta)\mu}_{q,a+}t^{k-\gamma}(a/t;q)_{k-\gamma}\right)(x)&=&\left(I^{\gamma-\nu}_{q,a+}
D^{\gamma}_{q,a+}t^{\gamma-k}(a/t;q)_{\gamma-k}\right)(x)=0
\end{eqnarray*}
and  using (\ref{additive3.2}) we get that
\begin{eqnarray*}
\left(D^{(\alpha,\beta)\mu}_{q,a+}I^\nu_{q,a+}f\right)(x)&=&\left( I_{q,a+}^{\gamma-\nu}D^{\gamma}_qI^\nu_{q,a+}f(\cdot,y(\cdot))\right)(x)\\
&=&\left(I_{q,a+}^{\gamma-\nu}I^{\nu-\gamma}_{q,a+}f(\cdot,y(\cdot))\right)(x)\\
&=&f(x,y(x)).
\end{eqnarray*}

Therefore,  we can rewrite (\ref{additive3.8}) in the form 
\begin{eqnarray*}
\left(D^{(\alpha,\beta)\mu}_{q,a+}y\right)(x)&=& f(x,y(x)).
\end{eqnarray*}

Finally, we will show that the initial condition of (\ref{additive3.4}) also holds. For this aim, we apply the operator $I^{(1-\mu)(n-\beta)}_{q,a+}$ to both sides of (\ref{additive3.5}) and using  (\ref{additive2.2}), (\ref{additive2.4}), we conclude that
\begin{eqnarray}\label{additive3.10}
\left(I^{(1-\mu)(n-\beta)}_{q,a+}y\right)(x)&=&\sum\limits_{k=1}^n\frac{\xi_k}{\Gamma_q\left(\gamma-k+1\right)}
\left(I^{(1-\mu)(n-\beta)}_{q,a+}t^{\gamma-k}(a/t;q)_{\gamma-k}\right)(x)\nonumber\\
&+& \left[I^{(1-\mu)(n-\beta)}_{q,a+}I^\nu_{q,a+}f(\cdot,y(\cdot))\right](x)\nonumber\\
&=&\sum\limits_{k=1}^n\frac{\xi_k}{[n-k]_q!} x^{n-k}(a/x;q)_{n-k}+ \left(I^{\mu(\alpha-n)-n}_{q,a+}f(\cdot,y(\cdot))\right)(x) .
\end{eqnarray}

Let  $0\leq m\leq n-1$. Then, we apply the operator $D^m_q$ to both sides of (\ref{additive3.10}) by using   (\ref{additive2.3}) and (\ref{additive2.6}),   to obtain  that
\begin{eqnarray}\label{additive3.11}
\left(D^m_qI^{(1-\mu)(n-\beta)}_{q,a+}y\right)(x)&=&\sum\limits_{k=1}^n\frac{\xi_k}{[n-k]_q!} D^m_q\left[x^{n-k}(a/x;q)_{n-k}\right]+\left( D^m_qI^{\mu(\alpha-n)-n}_{q,a+}f(\cdot,y(\cdot))\right)(x)\nonumber\\
&=&\sum\limits_{k=1}^n\frac{\xi_k}{[n-k-m]_q!}  x^{n-k-m}(a/x;q)_{n-k-m}\nonumber\\
&+&\left(I^{\mu(\alpha-n)-n-m}_{q,a+}f(\cdot,y(\cdot))\right)(x).
\end{eqnarray}

Taking in (\ref{additive3.11}) the limit  $ x \rightarrow a+$, we get the relations in (\ref{additive3.4}). Thus also the sufficiency is proved, which completes the
proof of Theorem \ref{thm3.3}.  
\end{proof}

\vspace{0,3cm}
In the next theorem,  we give conditions for a unique global solution to the Cauchy-type
problem (\ref{additive3.3})-(\ref{additive3.4}) in the space $L^1_{\alpha,\beta,\mu, q}[a,b]$ defined for $\alpha>0$ by
\begin{eqnarray*}
L^1_{\alpha,\beta,\mu, q}[a,b]:=\left\{y\in L_q^1[a, b]: D^{(\alpha,\beta)\mu}_{q,a+}y \in L_q^1[a, b]\right\}.
\end{eqnarray*}

The proof of the following existence and uniqueness theorem depends heavily on Theorem \ref{thm3.3} and the
Banach's fixed point Theorem (see  \cite[Subsection 8.1]{KF1968}).

\begin{thm}\label{thm3.4} We assume $\gamma=\beta+\mu(n-\beta)$ and $\nu=\beta+\mu(\alpha-\beta)$ for  $n-1<\alpha,\beta \leq n, n\in \mathbb{N}$ and $0\leq \mu\leq 1$. Let $n-1<\alpha\leq n; n\in\mathbb{N}$, and let     $f:
[a,b]\times \mathbb{R} \rightarrow \mathbb{R}$ be a function such that $f(\cdot,v(\cdot))\in L_q^1[a,b]$ for all $v\in L_q^1[a,b]$,  and it
satisfies the Lipschitz condition  in the following form:
\begin{eqnarray}\label{additive3.12}
\left|f(x,v_1(x))-f(x,v_2(x))\right|\leq A\left|v_1(x)-v_2(x)\right|,
\end{eqnarray}
where $A>0$ does not depend on $x \in[a,b]$ and $v_1,v_2\in L_q^1[a,b]$.   Then there exists a unique solution $y\in L^1_{\alpha,\beta,\mu, q}[a,b]$ to the Cauchy-type problem (\ref{additive3.3})-(\ref{additive3.4}).
\end{thm}
\begin{proof} First we show that there exists a unique solution $y \in L_q^1[a, b]$. In view of  Theorem \ref{thm3.3}  it is sufficient to prove the existence of a unique solution $y \in L_q^1[a, b]$ of the nonlinear Volterra $q$-integral equation (\ref{additive3.5}). We assume that $[a,a_1] \subset [a,b]$  are such that
\begin{eqnarray}\label{additive3.13}
\omega_1:=AK\frac{\left(a_1-qa\right)_q^\alpha}{\Gamma_q\left(\alpha+1\right)}<1,
\end{eqnarray}
where $K$ is as in (\ref{additive2.7}). Obviously, the equation (\ref{additive3.5}) holds on the interval $[a,a_1]$. Consequently, we rewrite this equation in the form
  $y = Ty$ with 
\begin{eqnarray}\label{additive3.14}
\left(Ty\right)(x):=y_0+\left(I^\nu_{q,a+}f(\cdot,y(\cdot))\right)(x),
\end{eqnarray}
where $y_0:=\sum\limits_{k=1}^n\frac{\xi_k}{\Gamma_q\left(\gamma-k+1\right)}x^{\gamma-k}(a/x;q)_{\gamma-k}$.

If $y\in L^1_q[a,a_1]$, then $Ty\in L^1_q[a,a_1]$.  Moreover, for any  $y_1,y_2\in L^1_q[a,a_1]$, $f \left(\cdot,  y_j(\cdot)\right) \in L^1_q[a,a_1], \;j=1,2$, and using (\ref{additive2.7}), (\ref{additive3.12}) and (\ref{additive3.13}) we have that 
\begin{eqnarray*}
\| Ty_1-Ty_2\|_{L_q^1[a,a_1]}&=& \|I^\nu_{q,a+}\left[f(t,y_1(t))-f(t,y_2(t))\right]\|_{L_q^1[a,a_1]}\\
&\leq& A\|I^\nu_{q,a+} \left[y_1-y_2\right|\|_{L_q^1[a,a_1]}\\
&\leq& \omega_1\| y_1-y_2 \|_{L_q^1[a,a_1]},
\end{eqnarray*}
which completes the proof of our first step.   The space $L_q^1[a,a_1]$ is the Banach space and according to the Banach's fixed point Theorem (see  \cite[Subsection 8.1]{KF1968}), there exists a unique solution $\widetilde{y}\in L_q^1[a,a_1]$ such that $T\widetilde{y}=\widetilde{y}'$ for $t\in[a,a_1]$.

Therefore, the solution $v\in L_q^1[a,a_1]$ is obtained from a limit of the convergent sequence $\left\{\left(T^i\widetilde{y}_0\right)\right\}$:
\begin{eqnarray}
\lim\limits_{i\rightarrow0}\|T^i\widetilde{y}_0-\widetilde{y}\|_{L_q^1[a,a_1]}=0,
\end{eqnarray}
where $ y_0$ is any function in $  L_q^1[a,a_1]$.

If at least one $\xi_k\neq0$ in  the initial condition (\ref{additive3.4}), then we deduce that  $\widetilde{y}_0=y_0$.  From (\ref{additive3.14}), the sequence $\left\{\left(T^i\widetilde{y}_0\right)(x)\right\}$ can be written as 
$$
\left(T^i\widetilde{y}_0\right)(x)=y_0+\left[I^\alpha_{q,a+}f(\cdot,T^{i-1}\widetilde{y}_0(\cdot))\right](x),
$$
for $i\in\mathbb{N}$. Let $y_i =T^i\widetilde{y}_0$. If we write 
$y_i(x)=y_0+\left(I^\alpha_{q,a+}f(\cdot,y_{i-1} (\cdot))\right)(x)$, then it is clear that
\begin{eqnarray*}
\lim\limits_{i\rightarrow0}\|y_i-\widetilde{y}\|_{L_q^1[a,a_1]}=0.
\end{eqnarray*}

Hence, we actually used the method of successive approximations to get a unique solution $y'$ to the integral
equation (\ref{additive3.5}) on $[a,a_1]$. 

Next, we consider the interval $[a_1, a_2]$ and $a_2= a_1 + h_1$,   $a_2<b$,  such that
\begin{eqnarray*} 
\omega_2:=AK\frac{\left(a_2-qa_1\right)_q^\alpha}{\Gamma_q\left(\alpha+1\right)}<1,
\end{eqnarray*}
for $ h_1>0$. Then (\ref{additive3.5}) can be rewritten as   
\begin{eqnarray*} 
y(x)=y_0(x)+ \left(I^\alpha_{q,a+}f(\cdot,y(\cdot))\right)(a_1)+\left(I^\alpha_{q,a_1+}f(\cdot,y(\cdot))\right)(x).
\end{eqnarray*}

Since the function $y$ is uniquely given on   $[a, a_1]$, the first integral can be considered as the known function,
and we can rewrite the last equation as
\begin{eqnarray*}
\left(Ty\right)(x):=y_{10}+\left(I^\alpha_{q,a_1+}f(\cdot,y(\cdot))\right)(x),
\end{eqnarray*}
where $y_{10}:=y_0(x)+ \left(I^\alpha_{q,a+}f(\cdot,y(\cdot))\right)(a_1)$. We note that $y_{10}\in L^1_q[a_1, a_2]$.

By using the same arguments as above, we derive that there exists a unique solution $y' \in L^1_q[a_1, a_2]$ to the equation (\ref{additive3.5}) on the interval $[a_1, a_2]$. We get the next interval $[a_2, a_3]$, where $a_3 = a_2 + h_2$ such that $ a_3 <b$ ($h_2 > 0$) and repeating the above process, we obtain a unique solution  $y' \in L^1_q[a, b]$  to the equation (\ref{additive3.5})  and hence to the Cauchy-type problem (\ref{additive3.3})-(\ref{additive3.4}).

Finally,    we must show that such a  solution $y\in  L^1_{\alpha,\beta,\mu, q}[a,b]$ is unique. According to the above proof, the solution $y\in L^1_q[a, b]$ is a limit of the sequence $\left\{y_i\right\}\in L^1_q[a, b]$:
\begin{eqnarray}\label{additive3.18}
\lim\limits_{i\rightarrow\infty} \|y_i-y\|_{L^1_q[a, b]}=0,
\end{eqnarray}
with the choice of certain $y_i$ on each $[a, a_1], [a_1, a_2], \dots , [a_{L-1}, b]$,  for $L\in\mathbb{N}$. Since
\begin{eqnarray*}
\|D^{\alpha,\beta}_{q,a+}y_i-D^{(\alpha,\beta)\mu}_{q,a+}y\|_{L^1_q[a, b]}&=&\|f\left(\cdot,y_i(\cdot)\right)-f\left(\cdot,y(\cdot)\right)\|_{L^1_q[a, b]}\\
&\leq& A
\| y_i- y\|_{L^1_q[a, b]},
\end{eqnarray*}
by (\ref{additive3.18}), we get that
\begin{eqnarray*}
\lim\limits_{i\rightarrow\infty}\|D^{(\alpha,\beta)\mu}_{q,a+}y_i-D^{(\alpha,\beta)\mu}_{q,a+}y\|_{L^1_q[a, b]}=0,
\end{eqnarray*}
and hence $D^{\alpha,\beta}_{q,a+}\in L^1_q[a, b]$. This completes the proof.
\end{proof}

\section{Miscellaneous Examples}\label{section4}

In Section \ref{section3},  we have introduced sufficient conditions for the Cauchy-type problem  to have a unique solution in some subspaces of the space of $q$-integrable functions.  In fact, in the classical case, these conditions show that the $q$-difference equations of fractional order also can have integrable solutions, provided that their right-hand sides are $q$-integrable. However,  these conditions are not sufficient. Below we present two examples and discuss these examples in connection with the results obtained in Section \ref{section3}  \cite[See examples 3.1-3.2 for fractional cases]{KST2006}.

Let $\gamma=\beta+\mu(n-\beta)$ and $\nu=\beta+\mu(\alpha-\beta)$ for  $n-1<\alpha,\beta \leq n, n\in \mathbb{N}$ and $0\leq \mu\leq 1$.

\begin{example}   We assume that $I^\gamma_{q,qa+}y \in AC_q^n[qa, b]$ and $\lambda, \delta \in \mathbb{R}^+$. Then the $q$-differential equation
\begin{eqnarray}\label{additive4.1}
\left(D^{(\alpha,\beta)\mu}_{q,qa+}\right)(x)&=&  \lambda \frac{(x-q^{-\nu-\delta+1}a)_q^{\nu+\delta}}{(x-q^{-2\nu-\delta+1}a)_q^\nu}  y^2(x) ,
\end{eqnarray}
with the initial condition
\begin{eqnarray}\label{additive4.2}
\lim\limits_{x\rightarrow qa+}\left(D^k_qI^{(1-\mu)(n-\beta)}_{q,qa+}y \right)(x)=0
\end{eqnarray}
has a unique solution in the following form: 
\begin{eqnarray}\label{additive4.3}
y(x)=\frac{1}{\lambda} \frac{\Gamma_q\left(1-\alpha-\gamma\right)}{\Gamma_q\left(1-2\alpha-\gamma\right)}
\left(x- qa\right)^{-\nu-\delta}_q,
\end{eqnarray}
for $-2\nu-\delta+1>0$, where $\left(x- a\right)^\alpha_q=x^\alpha(a/x;q)_\alpha$.

Indeed,  we consider
 $$
 f(x,y(x)):=\lambda   \frac{(x-q^{-\alpha-\delta+1}a)_q^{\alpha+\delta}}{(x-q^{-2\alpha-\delta+1}a)_q^\alpha}  y^2(x),
 $$
and assume that $I^\gamma_{q,a+}y \in AC_q^n[qa, b]$ and  $G:=\left\{qa\leq x\leq b;  \left|y(x)\right|<M<\infty\right\}$. For $-2\nu-\delta+1>0$,
$x\in[qa,b]$ and $y_1,y_2 \in G$,
\begin{eqnarray*}
\left|f(x,y_1(x))-f(x,y_2(x))\right|&=& \lambda \left|\frac{(x-q^{-\nu-\delta}a)_q^{\nu+\delta}}{(x-q^{-2\nu-\delta+1}a)_q^\nu}  y^2_1(x)\right.\\
&-&\left.\frac{(x-q^{-\nu-\delta})_q^{\nu+\delta}}{(x-q^{-2\nu-\delta+1}a)_q^\nu}  y^2_2(x) \right|\\
&\leq&\lambda\frac{(b-q^{-\nu-\delta+1}a)_q^{\nu+\delta}}{(a-q^{-2\nu-\delta+1}a)_q^\nu} \left|y_1 (x)+y_2(x)\right|\left|y_1 (x)-  y_2(x)\right|\\
&\leq&\lambda\frac{(b-q^{-\nu-\delta+1}a)_q^{\nu+\delta}}{(a-q^{-2\nu-\delta+1}a)_q^\nu}2M\left|y_1 (x)-  y_2(x)\right|,
\end{eqnarray*}
which proves that $f$ satisfies the Lipschitz condition (\ref{additive3.12}) in Theorem \ref{thm3.4}.

Therefore,
\begin{eqnarray*}
\int\limits_{qa}^b\left|f(x,y (x))\right|d_qx&=& \lambda \int\limits_{qa }^b\frac{(x-q^{-\nu-\delta+1}a)_q^{\nu+\delta}}{(x-q^{-2\nu-\delta+1}a)_q^\nu}\left|y(x)\right|^2 d_qx\\
&\leq& \lambda \frac{(b-q^{-\nu-\delta+1}a)_q^{\nu+\delta}}{(a-q^{-2\nu-\delta+1}a)_q^\nu} M^2,
\end{eqnarray*}
which means that $f(x,y (x))\in L_q^1[qa,b]$. We can conclude that, according to Theorem \ref{thm3.3}, the equation has a unique solution in $L_q^1[qa,b]$. Moreover, by applying the operator $D^{(\alpha,\beta)\mu}_{q,qa+}$ (see \ref{additive3.1})and using the Remark \ref{rem3.2} and  (\ref{additive2.2}), (\ref{additive2.3}), we get that
\begin{eqnarray*}
D^{(\alpha,\beta)\mu}_{q,qa+} \left(x- qa\right)^{-\nu-\delta}_q &=&  I_{q,qa+}^{\gamma-\nu}D^{\gamma}_{q,qa+}(x-qa)_q^{-\nu-\delta} \\
&=&\frac{\Gamma_q(1-\nu-\delta)}{\Gamma_q{(1-\nu-\delta-\gamma)}} \left[I_{q,qa+}^{\gamma-\nu}
 (x-qa)_q^{-\nu-\delta-\gamma} \right]\\
&=&  \frac{\Gamma_q(1-\nu-\delta )}{\Gamma_q(1-\delta-2\nu)}
 (x-qa)_q^{-2\nu-\delta}  
\end{eqnarray*}
for $2\nu+\delta<1$. Hence,
\begin{eqnarray}\label{additive4.4}
D^{(\alpha,\beta)\mu}_{q,qa+} \left[y(x)\right]
=\frac{1}{\lambda}\left[\frac{\Gamma_q(1-\nu-\delta )}{\Gamma_q(1-\delta-2\nu)}\right]^2
\left(x- qa\right)^{-2\nu-\delta}_q.
\end{eqnarray}

By using the well-known formulas (see  \cite[formula (3.7)]{CK2000}):  
$$
(x-q^{-\nu-\delta+1}a)_q^{\nu+\delta}(x-qa)_q^{-\nu-\delta}=1
$$
and
$$
 (x-qa)_q^{-\nu-\delta}=(x-qa)_q^{-2\nu-\delta}(x-q^{-2\nu-\delta}a)_q^\nu,
 $$
 we have that
\begin{eqnarray}\label{additive4.5}
 f(x,y(x))&=&\lambda   \frac{(x-q^{-\nu-\delta+1}a)_q^{\nu+\delta}}{(x-q^{-2\nu-\delta+1}a)_q^\nu}  y^2(x)\nonumber\\
&=&\frac{1}{\lambda}\left[\frac{\Gamma_q\left(1-\nu-\delta\right)}{\Gamma_q\left(1-2\nu-\delta\right)}\right]^2
\times\frac{(x-q^{-\nu-\delta+1}a)_q^{\nu+\delta}}{(x-q^{-2\alpha-\gamma+1}a)_q^\alpha}
\left[(x-qa)_q^{-\nu-\delta}\right]^2\nonumber\\
&=&\frac{1}{\lambda}\left[\frac{\Gamma_q\left(1-\nu-\delta\right)}{\Gamma_q\left(1-2\nu-\delta\right)}\right]^2
 (x-qa)_q^{-2\nu-\delta}=D^{(\alpha,\beta)\mu}_{q,qa+}\left[y(x)\right].
\end{eqnarray}

Now, by combining (\ref{additive4.4}) and using   (\ref{additive4.5}) and  $\lim\limits_{x\rightarrow qa}\left(x- qa\right)^{-\nu-\delta}_q=0$ for $x\in[qa,b]$,   we see that $y(x)$ defined by (\ref{additive4.3}) satisfies (\ref{additive4.1})-(\ref{additive4.2}). The proof of our claim is complete.

\end{example}

\begin{example} Let $\delta, \lambda\in\mathbb{R}^+$.  Then the  $q$-differential equation 
\begin{eqnarray}\label{additive4.6}
\left(D^{(\alpha,\beta)\mu}_{q,a+}y\right)(x) &=&  \lambda  \left[\frac{(x-qa )_q^{2\nu+2\delta}}
{(x-q^{\nu+2\delta+1}a )_q^\nu}\right]^\frac{1}{2}\left[y(x)\right]^\frac{1}{2},
\end{eqnarray}
with the initial condition
\begin{eqnarray}\label{additive4.7}
\lim\limits_{x\rightarrow qa+}\left(D^k_qI^{(1-\mu)(n-\beta)}_{q,qa+}y \right)(x)&=&0
\end{eqnarray}
has a unique solution in the following form:
\begin{eqnarray}\label{additive4.8}
y(x)=\left[\lambda  \frac{\Gamma_q\left(\nu+2\delta+1\right)}{\Gamma_q\left(2\nu+2\delta+1\right)}
\right]^2\left(x-qa\right)^{2\nu+2\delta}_q.
\end{eqnarray}

Since  $\left(x-qa\right)^{2\nu+2\delta}_q=\left(x-qa\right)^{\nu+2\delta}_q\left(x-q^{\nu+2\delta}a\right)^\nu_q$ we get 
\begin{eqnarray}\label{additive4.9}
f(x,y(x))&:=& \lambda   \frac{\left[(x-qa )_q^{2\nu+2\delta}\right]^\frac{1}{2}}
{(x-q^{\nu+2\delta+1}a )_q^\nu} \left[y(x)\right]^\frac{1}{2}\nonumber\\
&=&\lambda^2   \frac{\Gamma_q\left(\nu+2\delta+1\right)}{\Gamma_q\left(2\nu+2\delta+1\right)}\left[\frac{(x-qa )_q^{2\nu+2\delta}}
{(x-q^{\nu+2\delta+1}a )_q^\nu}\right]^\frac{1}{2}
\left[\left(x-qa\right)^{2\nu+2\delta}_q\right]^\frac{1}{2}\nonumber\\
&=&\lambda^2   \frac{\Gamma_q\left(\nu+2\delta+1\right)}{\Gamma_q\left(2\nu+2\delta+1\right)} \left(x-qa\right)^{\nu+2\delta}_q. 
\end{eqnarray}

We assume that $I^\gamma_{q,a+}y \in AC_q^n[a, b]$ and 
$  G:=\left\{a\leq x\leq b;  M_1< y(x) <M_2,   0<M_1,M_2\right\}$. Then, for $x\in[a, b]$ and $y_1, y_2\in G$ we obtain
\begin{eqnarray*}
|f(x,y_1(x))-f(x,y_2(x))|&\leq&\lambda\frac{\left[(b-qa )_q^{2\nu+2\delta}\right]^\frac{1}{2}}
{(a-q^{\nu+2\delta+1}a )_q^\nu }\left|y^\frac{1}{2}_1(x)-y^\frac{1}{2}_2(x)\right|\\
&=&\lambda\frac{\left[(b-qa )_q^{2\nu+2\delta}\right]^\frac{1}{2}}
{(a-q^{\nu+2\delta+1}a )_q^\nu }\frac{\left|y_1(x)-y_2(x)\right|}{y^\frac{1}{2}_1(x)+y^\frac{1}{2}_2(x)}\\
&\leq&\frac{\lambda}{2M_1}\frac{\left[(b-qa )_q^{2\nu+2\delta}\right]^\frac{1}{2}}
{(a-q^{\nu+2\delta+1}a )_q^\nu}\left|y_1(x)-y_2(x)\right|,
\end{eqnarray*}
which proves that $f$ satisfies the Lipschitz condition (\ref{additive3.6}) in Theorem \ref{thm3.4}.

Moreover,
\begin{eqnarray*}
\int\limits_{qa}^b\left|f(x,y (x))\right|d_qx
 \leq  \lambda \frac{\left[(b-qa )_q^{2\nu+2\delta}\right]^\frac{1}{2}}
{(a-q^{\nu+2\delta+1}a )_q^\delta}  M^\frac{1}{2},
\end{eqnarray*}
which means that $f(x,y (x))\in L_q^1[qa,b]$. We can conclude that according to the Theorem \ref{thm3.3}, the equation has a unique solution in $L_{q}^1[qa,b]$.  Moreover, by applying the operator $D^{(\alpha,\beta)\mu}_{q,qa+}$ (see \ref{additive3.1}) to (\ref{additive4.8})  and using   Remark \ref{rem3.2} and  (\ref{additive2.2}), (\ref{additive2.3}), we get that
\begin{eqnarray*}
D^{(\alpha,\beta)\mu}_{q,qa+} \left(x- qa\right)^{2\nu+2\delta}_q &=&  I_{q,qa+}^{\gamma-\nu}D^{\gamma}_{q,qa+}(x-qa)_q^{2\nu+2\delta} \\
&=&\frac{\Gamma_q(1+2\nu+2\delta)}{\Gamma_q{(1+2\nu+2\delta-\gamma)}} \left[I_{q,qa+}^{\gamma-\nu}
 (x-qa)_q^{2\nu+2\delta-\gamma} \right]\\
&=&\frac{\Gamma_q(1+2\nu+2\delta)}{\Gamma_q{(1+\nu+2\delta)}}
(x-qa)_q^{\nu+2\delta }  
\end{eqnarray*}
for $2\nu+\delta<1$. Hence, by \ref{additive4.9} 
\begin{eqnarray*} 
D^{(\alpha,\beta)\mu}_{q,qa+} \left[y(x)\right]
=  \lambda^2 \frac{\Gamma_q(1+2\nu+2\delta)}{\Gamma_q{(1+\nu+2\delta)}}
(x-qa)_q^{\nu+2\delta }=f(x,y(x)). 
\end{eqnarray*}

Now, by combining (\ref{additive4.4}) and using that  (\ref{additive4.5}) and  $\lim\limits_{x\rightarrow qa}\left(x- qa\right)^{-\nu-\delta}_q=0$ for $x\in[qa,b]$,   we see that $y(x)$ defined by (\ref{additive4.3}) satisfies (\ref{additive4.6})-(\ref{additive4.7}). The proof of our statement is complete.
\end{example}

\section{A linear Cauchy-type $q$-fractional problem}

Let  $\gamma=\beta+\mu(n-\beta)$ and $\nu=\beta+\mu(\alpha-\beta)$ for  $n-1<\alpha,\beta \leq n, n\in \mathbb{N}$ and $0\leq \mu\leq 1$.

The (Mittag-Leffler) $q$-function $E_{\alpha,\beta}\left(z;q\right)$ is defined by (\cite{SPT2020})
\begin{eqnarray}\label{additive5.1}
E_{\alpha,\beta}\left[\lambda x^{\gamma}(a/x;q)_{\gamma};q\right]=
\sum\limits_{k=0}^{\infty}\frac{\lambda^{k}x^{k\gamma}(a/x;q)_{k\gamma}}{\Gamma_q(\alpha k+\beta)}. 
\end{eqnarray}

We study the following Cauchy type $q$-fractional problem  
\begin{eqnarray}\label{additive5.2}
\left(D^{(\alpha,\beta)\mu}_{q,a+}y\right)(x) -\lambda y(x)&=&f(x), 
\end{eqnarray}
\begin{eqnarray}\label{additive5.3}
\lim\limits_{x\rightarrow a+}\left(D^k_qI^{(1-\mu)(n-\beta)}_{q,a+}y \right)(x)=\xi_k,  \;\;\;\xi_k\in\mathbb{R},  k=0,1,2,\dots n-1.
\end{eqnarray}

\begin{thm}\label{thm5.1}
Let  $\gamma=\beta+\mu(n-\beta)$ and $\nu=\beta+\mu(\alpha-\beta)$ for  $n-1<\alpha,\beta \leq n, n\in \mathbb{N}$ and $0\leq \mu\leq 1$. We assume that  $\lambda \in \mathbb{R}$ is  such that
\begin{eqnarray*}
|\lambda|b^{\nu}(1-q)^{\nu}<1,
\end{eqnarray*}
and $f\in L^1_q[a,b]$. Then the Cauchy problem   (\ref{additive5.2})-(\ref{additive5.3}) has a unique solution
$y \in L^1_{\alpha,\beta,\mu, q}[a,b]$ and  this solution is given by

\begin{eqnarray}\label{additive5.4}
y(x)&=&\sum\limits_{k=1}^{n} b_k x^{\nu-k}E_{\nu,\nu-k+1}\left[\lambda x^\nu;q\right] \nonumber\\
&+&\int\limits_0^{x} x^{\nu-1}\left(qt/x;q\right)_{\nu-1}   E_{\nu,\nu}\left[\lambda{x^\nu}(q^\nu{t}/x;q)_\nu;q\right] f(t)d_qt.
\end{eqnarray}
\end{thm}

\begin{proof}     First, we solve the Volterra q-integral equation (\ref{additive3.5}) for 
\begin{eqnarray*}
f(x,y(x)):=\lambda{y(x)}+f(x).
\end{eqnarray*}

For this aim, we apply the method of successive approximations by setting
$$
y_{0}(x)=\sum\limits_{k=1}^n\frac{\xi_k}{\Gamma_q\left(\gamma-k+1\right)}x^{\gamma-k}(a/x;q)_{\gamma-k}
$$
and
\begin{eqnarray}\label{additive5.5}
y_i(x)&=&y_0(x)+
\frac{\lambda{x^{\nu-1}}}{\Gamma_q\left(\nu\right)}\int_{0}^{x}(qt/x;q)_{\nu-1}y_{i-1}(t)d_qt\nonumber\\ &+&\frac{x^{\nu-1}}{\Gamma_q\left(\nu\right)}\int_{0}^{x}(qt/x;q)^{\nu-1}_{q}f(t)d_qt.
\end{eqnarray}

Using   (\ref{additive2.1}) and   (\ref{additive5.5})  we find $y_1(x)$:

$$
y_1(x)=y_0(x)+\lambda\left(I^{\nu}_{q,0+}y_0\right)(x)+\left(I^{\nu}_{q,0+}f\right)(x),
$$
that is,

\begin{eqnarray}\label{additive5.6}
y_{1}(x)&=&\sum\limits_{k=1}^{n}\frac{b_k}{\Gamma_q\left(\nu-k+1\right)}x^{\nu-k}+
\lambda\sum\limits_{k=1}^{n}\frac{b_k}{\Gamma_q\left(\nu-k+1\right)}\left(I^\nu_{q,0+}t^{\nu-k}\right)(x)+
\left(I^{\nu}_{q,0+}f\right)(x)\nonumber\\
&=&\sum\limits_{k=1}^{n}\frac{b_k}{\Gamma_q\left(\nu-k+1\right)}x^{\nu-k}+
\lambda\sum\limits_{k=1}^{n}\frac{b_k x^{2\nu-k}}{\Gamma_q\left(2\nu-k+1\right)}+\left(I^{\nu}_{q,0+}f\right)(x) \nonumber\\
&=&\sum\limits_{k=1}^{n} b_k \sum\limits_{m=1}^2\frac{\lambda^{m-1}x^{m\nu-k}}{\Gamma_q\left(\nu m-k+1\right)}+\left(I^{\nu}_{q,0+}f\right)(x).
\end{eqnarray}

Similarly, using   (\ref{additive2.1}),  (\ref{additive2.2}), (\ref{additive2.4})  and (\ref{additive5.6})  we have for $y_2(x)$ that
\begin{eqnarray*}
y_2(x)&=&y_0(x)+\lambda \left(I^\nu_{q,0+}y_1\right)(x)+\left(I^\nu_{q,0+}f\right)(x) \\
 &=&\sum\limits_{k=1}^{n}\frac{b_k}{\Gamma_q\left(\nu-k+1\right)}x^{\nu-k}\\
 &+&\frac{\lambda}{\Gamma_q\left(\nu\right)} \sum\limits_{k=1}^{n} b_k \sum\limits_{m=1}^2\frac{\lambda^{m-1}}{\Gamma_q\left(\nu m-k+1\right)}\left(I^\nu_{q,0+}t^{m\nu-k}\right)(x)\\
&+&\lambda\left(I^{\nu}_{q,0+}I^{\nu}_{q,0+}f(t)\right)(x)+\left(I^{\nu}_{q,0+}f\right)(x)\\
 &=&\sum\limits_{k=1}^{n}\frac{b_k}{\Gamma_q\left(\nu-k+1\right)}x^{\nu-k}+\lambda\sum\limits_{k=1}^{n} b_k \sum\limits_{m=1}^2\frac{\lambda^{m-1}}{\Gamma_q\left(\nu(m+1)-k+1\right)}x^{\nu(m+1)-k} \\
&+&\lambda\left(I^{2\nu}_{q,0+}f(t)\right)(x)+\left(I^{\nu}_{q,0+}f\right)(x)
\end{eqnarray*}
\begin{eqnarray*}
 &=&\sum\limits_{k=1}^{n}\frac{b_k}{\Gamma_q\left(\nu-k+1\right)}x^{\nu-k}+\lambda\sum\limits_{k=1}^{n} b_k \sum\limits_{m=1}^2\frac{\lambda^{m-1}}{\Gamma_q\left(\nu(m+1)-k+1\right)}x^{\nu(m+1)-k} \\
&+&\frac{\lambda{x^{2\nu-1}}}{\Gamma\left(2\nu\right)}\int\limits_0^xf(t)\left(qt/x;q\right)_{2\nu-1}d_qt+\left(I^{\nu}_{q,0+}f\right)(x).
\end{eqnarray*}

Thus,
\begin{eqnarray*}
y_2(x)&=&\sum\limits_{k=1}^{n}b_k\sum\limits_{m=1}^3
\frac{\lambda^{m-1}x^{\nu{m}-k}}{\Gamma_q\left(\nu{m}-k+1\right)}\nonumber\\
&+&\int_0^x\left[\sum\limits_{m=1}^{2}\frac{\lambda^{m-1}x^{\nu{m}-1}(qt/x;q)_{\nu{m}-1}}{\Gamma_q(\nu{m})}\right] f(t)d_qt.
\end{eqnarray*}

Continuing this process, we derive the following relation for $y_i(x)$:
\begin{eqnarray*}
 y_i(x)&=&\sum\limits_{k=1}^{n}b_k\sum\limits_{m=1}^{i+1}
\frac{\lambda^{m-1}x^{\nu{m}-k}}{\Gamma_q\left(\nu{m}-k+1\right)}\nonumber\\
&+&\int_{0}^{x}\left[\sum\limits_{m=1}^i\frac{\lambda^{m-1}x^{\nu{m}-1}(qt/x;q)_{\nu{m}-1}}{\Gamma_q(\nu{m})}\right] f(t)d_qt\nonumber\\
&=&\sum\limits_{k=1}^{n}b_k\sum\limits_{m=0}^{i}
\frac{\lambda^{m}x^{\nu{(m+1)}-k}}{\Gamma_q\left(\nu{(m+1)}-k+1\right)}\nonumber\\
&+&\int_{0}^{x}\left[\sum\limits_{m=0}^{i-1}\frac{\lambda^{m}x^{\nu{(m+1)}-1}(qt/x;q)_{\nu{(m+1)}-1}}{\Gamma_q(\nu(m+1))}\right] f(t)d_qt.
\end{eqnarray*}

From (\ref{additive2.1C}) and (\ref{additive2.1D}) its follows that 
\begin{eqnarray*}
\frac{\lambda^{m}x^{\nu{m}}({qt}/{x};q)_{\nu{m}}}{\Gamma_q(\nu{m}+\nu)} &=&\left[\lambda{x}^\nu(1-q)^\nu\right]^m\frac{(q^{\nu(m+1)}; q)_\infty }{(q^{\nu{m}+1}t/x; q)_\infty }(1-q)^{\nu-1} \\
&\leq&\left[\lambda{x}^\nu(1-q)^\nu\right]^m\frac{(q^{\nu(m+1)}; q)_\infty }{(q^{\nu{m}+1} ; q)_\infty }(1-q)^{-1}\\
&\leq&\left[\lambda{b}^\nu(1-q)^\nu\right]^m (1-q)^{-1}.
\end{eqnarray*}

Thus, 
\begin{eqnarray*}
\left|\sum\limits_{m=0}^\infty\frac{\lambda^{m}x^{\nu{m}}({qt}/{x};q)_{\nu{m}}}{\Gamma_q(\nu{m}+\nu)}\right|&\leq&\frac{1}{1-q}\sum\limits_{m=0}^\infty\left[|\lambda|{b}^\nu(1-q)^\nu\right]^m\leq\frac{1}{|\lambda|{b}^\nu(1-q)^{\nu+1}}, 
\end{eqnarray*}
which means that this is absolutely and uniformly convergent for $\lambda{b}^\nu(1-q)^\nu<1$. Taking the limit as $i\rightarrow\infty$ and using the well-known formulas (see  \cite[formula (3.7)]{CK2000}), we obtain the following explicit solution $y(x)$ to the $q$-integral equation (\ref{additive5.4}):
\begin{eqnarray*}
y(x)&=&\sum\limits_{k=1}^{n}b_kx^{\nu-k}\sum\limits_{m=0}^\infty
\frac{\lambda^{m}x^{\nu{m}}}{\Gamma_q\left(\nu{m}+\nu-k+1\right)}\nonumber\\
&+&\int_{0}^{x}x^{\nu-1}({qt}/{x};q )_{\nu-1}\left[\sum\limits_{m=0}^\infty\frac{\lambda^{m}x^{\nu{m}}({qt}/{x};q)_{\nu{m}}}{\Gamma_q(\nu{m}+\nu)}\right] f(t)d_qt.
\end{eqnarray*}

On the basis of Theorem \ref{thm3.4} and (\ref{additive5.1}), the last expression represents the explicit solution to the Volterra $q$-integral equation (\ref{additive5.4}) and hence to the Cauchy-type problem (\ref{additive5.2})-(\ref{additive5.3}).
\end{proof}

\section{Conflict of Interests}
The authors declare that they have no conflict of interest.

\end{document}